\documentclass[reqno, 12pt]{amsart}

\pdfoutput=1

\usepackage{enumerate}
\usepackage{latexsym}
\usepackage[centertags]{amsmath}
\usepackage{amsfonts}
\usepackage{amssymb}
\usepackage{amsthm}
\usepackage{newlfont}
\usepackage{graphics}
\usepackage{color}
\usepackage{float}
\usepackage{diagbox}
\usepackage{hyperref}
\usepackage{longtable}
\usepackage{rotating}
\usepackage{multirow}

\usepackage{extarrows}

\usepackage[sort,compress,numbers]{natbib}

\usepackage[utf8]{inputenc}

\newtheorem{thm}{Theorem}[section]
\newtheorem{cor}[thm]{Corollary}
\newtheorem{lem}[thm]{Lemma}
\newtheorem{prop}[thm]{Proposition}

\theoremstyle{mydefinition}

\theoremstyle{myremark}

\newtheorem{exa}[thm]{Example}

\allowdisplaybreaks[4]

\def\CT{\mathop{\mathrm{CT}}}

\title[Young Tableaux with Periodic Walls]{Simple Generating Functions for Certain Young Tableaux with Periodic Walls}

\author[Feihu Liu and Guoce Xin]{Feihu Liu$^{1}$ and Guoce Xin$^{2, *}$}

\address{$^{1, 2}$School of Mathematical Sciences, Capital Normal University,
Beijing 100048, PR China}
\email{$^1$\texttt{liufeihu7476@163.com}\ \& $^2$\texttt{guoce\_xin@163.com}}
\date{January 26, 2024}
\thanks{$*$ This work is partially supported by NSFC(12071311).}
\begin{document}
\maketitle

\begin{abstract}
Recently, Banderier et. al. considered Young tableaux with walls, which are similar to standard Young tableaux, except that local decreases are allowed at some walls. We count the numbers $\overline{f}_m(n)$ of Young tableaux of shape $2\times mn$ with walls, that allow local decreases at the $(jm+i)$-th columns for all $j=0,\dots, n-1$ and $i=2,\dots, m$. We find that they have nice generating functions (thanks to the OEIS) as follows.
$$\overline{F}_m(x)=\sum_{n\geq 0}\overline{f}_m(n)x^n=\prod_{k=1}^{m}C(e^{k\frac{2\pi i}{m}} x^\frac{1}{m})=\exp \left(\sum_{n\geq 1}\binom{2mn-1}{mn-1}\frac{x^n}{n}\right),$$
where $C(x)=\frac{1-\sqrt{1-4x}}{2x}$ is the well-known Catalan generating function.
We prove generalizations of this result. Firstly, we use the Yamanouchi word to transform Young tableaux with horizontal walls into lattice paths. This results in a determinant formula. Then by lattice path counting theory,
we obtain the generating functions $F_r(x)$ for the number of lattice paths from $(0,0)$ to $(\ell n-r,kn)$ that never go above the path $(N^kE^{\ell})^{n-1}N^kE^{\ell-r}$, where $N,E$ stand for north and east steps, respectively. We also obtain exponential formulas for $F_1(x)$ and $F_\ell(x)$. The formula for $\overline{F}_m(x)$ is thus proved since it is just $F_1(x)$ specializes at $k=\ell=m$.
\end{abstract}

\noindent
\begin{small}
 \emph{Mathematic subject classification}: Primary 05A15; Secondary 05A10, 05E05.
\end{small}

\noindent
\begin{small}
\end{small}

\section{Introduction}
Throughout this paper, $\mathbb{C}$, $\mathbb{Z}$, $\mathbb{N}$ and $\mathbb{P}$ denote the set of all complex numbers, all integers, non-negative integers and positive integers, respectively.

Recently, Banderier et. al. \cite{C.Banderier18,C.Banderier21} considered a variation of standard Young tableaux,
called \emph{Young tableaux with walls}. We give a rigorous definition for our exploration.
Let $\lambda$ be a partition of $n$, denoted $\lambda \vdash n$. Its Young diagram is also
denoted $\lambda$.
See, e.g., \cite[Section 1.7]{RP.Stanley} for detailed definitions. In this paper, we focus on $\lambda=(m,m)$, so the Young diagram of $\lambda$
is a $2\times m$ rectangle with $|\lambda|=2m$ cells. Edges between two neighboring cells of $\lambda$ will be referred to as \emph{walls}.
A \emph{Young building} $\mathcal{B}$ of shape $\lambda$ is a pair $\mathcal{B}=(\lambda,W)$ where $W$ is a subset of the walls of $\lambda$.
The walls in $W$ are depicted as bold red edges.
A Young tableau with walls over $\mathcal{B}$ is a filling of the cells of $\lambda$
by labels $1,2,\dots, n$, such that each label appears exactly once and the labels are increasing along each row and column,
(conditions for standard Young tableaux)
but two labels separated by a wall need not be increasing. Denote by $YT(\lambda,W)$ the set of all such tableaux.
It reduces to the set of standard Young tableaux of shape $\lambda$ when $W$ is empty.
See Figure \ref{YoungT1} for an example of Young tableaux with walls of shape $\lambda=(6,6)$.
\begin{figure}[htp]
\centering
\includegraphics[width=4cm,height=1.5cm]{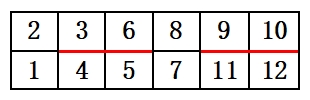}
\caption{A Young tableaux of shape $\lambda=(6,6)$ with walls.}
\label{YoungT1}
\end{figure}

Let $\mathcal{B}=(\lambda,W)$ be a rectangular building. That is, $\lambda$ is of rectangular shape.
Define $\mathcal{B}^n$ to be the Young building formed by connecting $n$ blocks of $\mathcal{B}$ horizontally. Then
Young tableaux over $\mathcal{B}^n$ are called \emph{Young tableaux with periodic walls}. Note that in this definition, $W$ is allowed to contain the rightmost edges of $\lambda$. Figure \ref{YoungT2} illustrates a Young building $\mathcal{B}_2$ on the left
and a Young tableau with walls in $YT(\mathcal{B}_2^3)$ on the right.

\begin{figure}[htp]
\centering
\includegraphics[width=9cm,height=2.5cm]{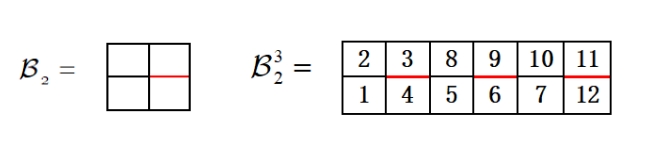}
\caption{A Young tableaux with periodic walls of a block $\mathcal{B}_2$ of shape $(2,2)$.}
\label{YoungT2}
\end{figure}

In \cite{C.Banderier21}, Banderier and Wallner counted Young tableaux with periodic walls by ``the density method".
They considered a classification of $2\times 2$ periodic shapes (a total of $2^6=64$ different models) and obtained counting formulas for $60$ models.
In particular, they
obtained the counting formula
$$\overline{f}_2(n):=|YT(\mathcal{B}_2^n)| =2^{2n+1}\texttt{Cat}(n)-\texttt{Cat}(2n+1),$$
where $\texttt{Cat}(n)$ is the well-known Catalan number with the Catalan generating function
\begin{equation}
C(x)=\sum_{n\geq 0}\texttt{Cat}(n)x^n= \frac{1-\sqrt{1-4x}}{2x}=\sum_{n\geq 0} \frac{1}{n+1}\binom{2n}{n} x^n.
\end{equation}
Note that $\overline{f}_2(n)$ is the sequence [A079489] in OEIS \cite{Sloane23}.

More generally, let $m$ be a positive integer and $\mathcal{B}_m$ be the Young building shape $\lambda=(m,m)$ with horizontal walls in all columns except the first column. See Figure \ref{YoungT3} (on the left) for two examples of elements in $YT(\mathcal{B}_3^2)$.
We are interested in the counting formula
of $\overline{f}_m(n)=|YT(\mathcal{B}_m^n)|$.

Computer experiment suggests that $\overline{f}_m(n)$ for $m=3,4,5,6$ are respectively the sequences [A213403], [A213404], [A213405] and [A213406] in OEIS \cite{Sloane23}. These sequences have nice generating functions, and can be summarized as the following result.
\begin{thm}\label{PeriodicWallGF}
Suppose $m$ is a positive integer.
Let $\overline{f}_m(n)$ denote the number of Young tableaux with periodic walls over $\mathcal{B}_m^n$, with convention $\overline{f}_m(0)=1$.
Then we have
$$\overline{F}_m(x)=\sum_{n\geq 0}\overline{f}_m(n)x^n=\prod_{k=1}^{m}C(\xi^{k} x^\frac{1}{m})=\exp \left(\sum_{n\geq 1}\binom{2mn-1}{mn-1}\frac{x^n}{n}\right),$$
where $\xi=e^{\frac{2\pi i}{m}}$ and $C(x)$ is the Catalan generating function.
\end{thm}
The motivation of this paper is to give a proof of this theorem.

To this end, we find several combinatorial interpretations of $\overline{f}_m(n)$, including a lattice path interpretation.
This allows us to prove Theorem \ref{PeriodicWallGF} using lattice path counting theory (e.g., \cite{UTamm02,RJChapman,JIrving09}). Indeed, our first result is the following.
\begin{thm}\label{MainResinIntrou}
Let $k,\ell\in\mathbb{P}$, $n\in \mathbb{N}$ and $1\leq r\leq \ell$. Let $f_r(n)$ be the number of lattice paths from $(0,0)$ to $(\ell n-r,kn)$ that never go above the path $(N^kE^{\ell})^{n-1}N^kE^{\ell-r}$, and let $w_1,...,w_{\ell}$ be the unique solutions of the equation $(w-1)^{\ell}-xw^{k+\ell}=0$ that are fractional power series.
Then the generating function $F_r(x)=\sum_{n\geq 0}f_r(n)x^n$ (with the convention $f_r(0)=1$) is given by
$$F_r(x)=\sum_{i=0}^{r-1}(-1)^i\binom{\ell-r+i}{i}e_{\ell-r+1+i}(w_1,w_2,...,w_{\ell}),$$
where $e_m(w_1,...,w_{\ell})$ is the $m$-th elementary symmetric function introduced in subsection 3.1.
\end{thm}
Specifically, when $r=\ell$, this result was obtained by de Mier and Noy (\cite{AnnadeMier}) in 2005.

For the two cases of $r=\ell$ and $r=1$, we obtain the following exponential formulas.
\begin{thm}\label{QXinPhdIntrod}
Let $q(n)$ be the number of lattice paths from $(0,0)$ to $(\ell n,kn)$ that never go above the path $(N^kE^{\ell})^n$. Let $Q(x)=\sum_{n\geq 0}q(n)x^n$. Then we have
$$Q(x)=\exp \left(\sum_{n\geq 1}\binom{kn+\ell n}{\ell n}\frac{x^n}{n}\right).$$
\end{thm}

\begin{thm}\label{FF1ExpIntrod}
Following the notation in Theorem \ref{MainResinIntrou}. Let $f_1(n)$ be the number of lattice paths from $(0,0)$ to $(\ell n-1,kn)$ that never go above the path $(N^kE^{\ell})^{n-1}N^kE^{\ell-1}$. Let $F_1(x)=\sum_{n\geq 0}f_1(n)x^n$. Then we have
$$F_1(x)=\exp \left(\sum_{n\geq 1}\binom{kn+\ell n-1}{\ell n-1}\frac{x^n}{n}\right).$$
\end{thm}

The paper is organized as follows.
In Section 2,
we consider Young tableaux with horizontal walls over a Young building $\mathcal{B}$, especially for $\mathcal{B}$ of shape $(m,m)$, a two-rowed partition.
By using the Yamanouchi word \cite[Propsition 7.10.3(d)]{RP.Stanley99}, such Young tableaux are encoded by
$y$-sequences, lattice paths, and reverse partitions. This allows us to give a determinant formula for
$YT(\mathcal{B})$ using lattice path counting theory. In particular, $\overline{f}_m(n)$ has a simple determinant formula in Proposition \ref{DeterFormuFM}.
Section 3 is devoted to the proof of our main result, Theorem \ref{MainResinIntrou}. In subsection 3.1 we introduce some basic knowledge about symmetric functions
and combinatorial sum identities that will be used in our proof. In subsection 3.2, we first introduce the functional equation system in \cite{AnnadeMier} obtained from Matroid Theory, then solve the system and complete the proof of Theorem \ref{MainResinIntrou}.
Section 4 focuses on the cases $r=1$ and $r=\ell$.  We prove Theorems \ref{QXinPhdIntrod} and \ref{FF1ExpIntrod} on their exponential formulas,
and complete the proof
of Theorem \ref{PeriodicWallGF}.
In Section 5, we discuss two other possible approaches to Theorem \ref{MainResinIntrou}. This leads to two byproducts in Theorem \ref{ThmFReccParit} and Theorem \ref{SFdetermidentify}.

\section{Rectangular Young Tableaux with Horizontal Walls}\label{YTWallsToPath}
Let $\mathcal{B}=(\lambda,W)$ be a rectangular Young building with horizontal walls. That is, $\lambda=(m,m,\dots, m)$, and
$W$ contains only horizontal walls. We use the Yamanouchi word \cite[Propsition 7.10.3(d)]{RP.Stanley99} to
transform Young tableaux with walls over $\mathcal{B}$ into several combinatorial objects, including lattice paths.
Though most of the ideas work for general $\lambda$, we focus on the two-row case, i.e., $\lambda=(m,m)$,
for its simplicity and its close relation with $f_m(n)$ in Theorem \ref{MainResinIntrou}.

For a  Young tableau $T$ with walls over $\mathcal{B}$ depicted in Figure \ref{YoungT4}, its labels satisfy
$x_1<x_2<\cdots$ and $y_1<y_2<\cdots$. Note that we did not mark walls in Figure \ref{YoungT4}.
Since $\mathcal{B}$ contains only horizontal walls, we may encode
$W$ as a subset $S$ of $[m]:=\{1,\dots,m\}$. Thus we have the extra conditions $x_i<y_i$ for all $i\not\in S$.
\begin{figure}[ht]
$\begin{array}{|c|c|c|c|c|}\hline
y_1 & y_2 & y_3 & \cdots & y_m \\ \hline
x_1 & x_2 & x_3 & \cdots & x_m \\ \hline
\end{array}$
\caption{\label{YoungT4}}
\end{figure}

Now we introduce several different encodings of $T\in YT((m,m),S)$.

\begin{enumerate}
\item The \emph{$y$-sequence} of $T$ is $y(T)=(y_1,y_2,\dots, y_m)$. The \emph{$x$-sequence} of $T$ is $x(T)=(x_1,x_2,\dots, x_m)$. Then $\{x_1,\dots,x_m\}=[2m]\setminus \{y_1,\dots, y_m\}$.

 \item The \emph{Yamanouchi word} of $T\in YT(\lambda,S)$ is defined by
$w(T)=w_1w_2\cdots w_{2m}$ with $w_i=\chi(i\ \text{in first row})$
where $\chi(true)=1$ and $\chi(false)=0$.

\item The lattice path of $T$ is defined by $p(T)= w(T)|_{0\to N, 1\to E}$, which
is the $NE$-sequence of a lattice path, with $N=(0,1)$ and $E=(1,0)$ standing for north step and east step, respectively.

\item The \emph{reverse partition} of $T$ is defined to be the partition $\mu(T)=(\mu_1,\mu_{2},\dots, \mu_m)$ lying above $p(T)$.
More precisely, $\mu_i\geq 0$ is the number of cells in the $(m-i+1)$-th row and to the left of $p(T)$. In some contexts, it is also referred to as the co-area sequence.
\end{enumerate}

For example, if $T$ is the figure in the bottom left corner of Figure \ref{YoungT3} with $S=\{2,3,5,6\}$, then the Yamanouchi word of $T$ is $011100011100$;
the $y$-sequence of $T$ is $y(T)=(2,3,4,8,9,10)$; the $x$-sequence of $T$ is $x(T)=(1,5,6,7,11,12)$; the path of $T$ is $p(T)=NE^3N^3E^3N^2$; and the reverse partition of $T$ is $\mu(T)=(0,3,3,3,6,6)$.
It should be evident that $\mu_i=x_i-i$.
\begin{figure}[htp]
\centering
\includegraphics[width=12cm,height=5cm]{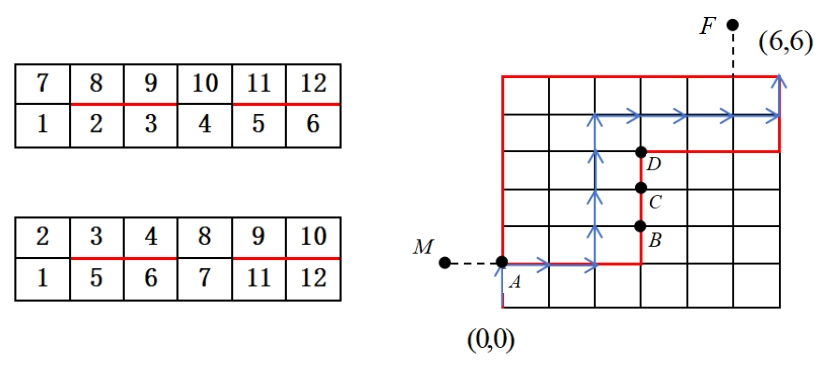}
\caption{An example of bijection for $\mathcal{B}_3^2$ in Proposition \ref{DeterFormuFM}.}
\label{YoungT3}
\end{figure}

The following lemma is immediate.
\begin{lem}\label{ComponentwisePath}
For any $T, T'$ in $YT((m,m),S)$ as above, the following are equivalent.
\begin{enumerate}
  \item $\mu(T)\leq \mu(T')$ componentwise, i.e.,
$\mu_i\leq \mu_i'$ for all $i$ if $\mu(T)=(\mu_1,\mu_2,\dots, \mu_m)$ and $\mu(T')=(\mu'_1,\mu'_2,\dots, \mu'_m)$.

  \item The path $p(T)$ lies above the path $p(T')$;

  \item The $y$-sequence $y(T)$ is larger than the $y$-sequence $y(T^{\prime})$ componentwise. We also denote $y(T)\leq y(T^{\prime})$
\end{enumerate}
We denote by $T\leq T'$ if one of the above conditions holds (hence all the conditions hold) true.
\end{lem}
The proof of the lemma is straightforward and we omit it.
Moreover, $\leq$ is a partial order on $YT((m,m),[m])$ and hence on its restriction $YT((m,m),S)$ for all $S$.
Indeed, it is the partial order induced by the containment partial order on their reverse partitions.

Two extreme cases of $S$ are worth mentioning: i) When $S$ is empty, $T\in YT((m,m),\emptyset)$ is a standard Young tableau, and its (classical) Yamanouchi word $w(T)$ corresponds to a \emph{Catalan path}, i.e., a lattice path that never
goes below the diagonal. The Yamanouchi map $\varphi: T\mapsto p(T)$ establishes a bijection from standard Young tableaux of shape $(m,m)$ to
Catalan paths of size $m$. ii) When $S=[m]$, one can verify that the map $\varphi$ establishes a bijection from $YT((m,m),[m])$ to the set of all lattice paths from
$(0,0)$ to $(m,m)$. These two cases invite us to find a good interpretation of the image of  $YT((m,m),S)$ under $\varphi$ for arbitrary
subset $S$.

To this end, we need the following lemma.
\begin{lem}\label{LemYminYmax}
Suppose $S$ is a subset of $[m]$. Then any $T\in YT((m,m),S)$ satisfies
$y_{\min}\leq y(T) \leq y_{\max}(S)$, where $y_{\min}=(m+1,m+2,\dots, 2m)$
and $y_{\max}(S)=(Y_1,\dots, Y_m)$ is determined by the following rule:
If $i\not\in S$ then $Y_i=2i$; otherwise $i\in S$ and $Y_i=i_0+i$ where $i_0$
is the largest $\ell<i$ satisfying $\ell\not\in S$ (if no such $\ell$ exists, then
we set $i_0=0$).
\end{lem}
\begin{proof}
Assume $y(T)=(y_1,\dots, y_m)$. Since $y_1<y_2<\cdots <y_m\leq 2m$ always hold,
$y_{\min}=(m+1,m+2,\dots, 2m)\leq y(T)$ for any $S$.

To see that $y(T) \leq y_{\max}(S)$, we first check that $(Y_1,\dots, Y_m)$
is a $y$-sequence of certain $T'\in YT((m,m),S)$. This is straightforward.
Next we claim that $y_i\geq Y_i$ for all $i$ so that the proof is completed.
We have to distinguish the two cases:
i) If $i\not\in S$, then $y_1<\cdots <y_i$ and $x_1<\cdots <x_i<y_i$. Thus $y_i\geq 2i=Y_i$;
ii) If $i\in S$, then $y_1<\cdots <y_i$ and $x_1<\cdots <x_{i_0}<y_{i_0}$.
Thus $y_i\geq i+i_0=Y_i$ as desired.
\end{proof}

\begin{prop}\label{TableauxLatticePartit}
Suppose $S$ is a subset of $[m]$. Let $T^0, T^M(S)$ be the tableau in $YT((m,m),S)$ with
$y$-sequence $y_{\min}$ and $y_{\max}(S)$ as in Lemma \ref{LemYminYmax}, respectively. Then
the following sets have the same cardinality.
\begin{enumerate}
  \item Tableaux $T$ in $YT((m,m),S)$ that satisfy $T^0 \leq T\leq T^M(S)$;

  \item Lattice paths from $(0,0)$ to $(m,m)$ that stay above $p(T^M(S))$;

  \item Partitions contained in $\mu(T^M(S))$.
\end{enumerate}
\end{prop}
\begin{proof}
The bijection between sets $(2)$ and $(3)$ is obvious, so it suffices to show that $T\mapsto p(T)$ is the desired bijection from set $(1)$ to set $(2)$. For a given tableaux $T$ in set $(1)$, the $y$-sequence of $T$ is $y(T)=(y_1,y_2,...,y_m)$. By Proposition \ref{LemYminYmax}, we have $y(T) \leq y_{\max}(S)$, i.e., $y_i\geq Y_i$. By Lemma \ref{ComponentwisePath}, the path $p(T)$ lies above the path $p(T^M(S))$. So $p(T)$ is well defined between sets $(1)$ and $(2)$. Moreover, $T\mapsto p(T)$ is an injection.

Given a path $P$ that stays above $p(T^M(S))$, we need to find a desired $T$ satisfying $p(T)=P$. Assuming that the $y$-sequence of $T^{M}(S)$ is $(Y_1,Y_2,...,Y_m)$, and the $x$-sequence of $T^{M}(S)$ is $(X_1,X_2,...,X_m)$. If $i\notin S$, then $X_i<Y_i$.
Let's construct a $T$ such that $p(T)=P$.
If the $r$-th step of $P$ is an $E$ step, then the entry $r$ appears in the first row of $T$. We mark it as $(y_1,y_2,...,y_m)$ in ascending order. So we have $y_i\geq Y_i$ for all $i$.  If the $r$-th step of $P$ is a $N$ step, then the entry $r$ appears in the second row of $T$. We mark it as $(x_1,x_2,...,x_m)$ in ascending order. So we have $x_i\leq X_i$ for all $i$. If $i\notin S$, then $x_i\leq X_i<Y_i\leq y_i$. So $T$ in $YT((m,m),S)$ that satisfy $T^0 \leq T\leq T^M(S)$. Furthermore, $p(T)=P$ and $p(T)$ is a bijection. This completes the proof.
\end{proof}

See Figure \ref{YoungT3} for an example. The first step of $p(T)$ is always $N$ since the $x_1$ in $\mathcal{B}_m^n$
has to be equal to $1$.

For the cardinality of $YT((m,m),S)$, there is a determinant formula by
the following result of Kreweras \cite{G.Kreweras} in lattice path counting theory.
\begin{lem}[\cite{G.Kreweras}]\label{DeterminParti}
Let $P_1=(a_1,a_2,...,a_n)$ and $P_2=(b_1,b_2,...,b_n)$ be two fixed reverse partitions, where $a_1\leq a_2\leq \cdots \leq a_n$ and $b_1\leq b_2\leq \cdots \leq b_n$. Suppose $P_2\leq P_1$, i.e., $b_i\leq a_i$ for $1\leq i\leq n$. The number of partitions lying between $P_2$ and $P_1$ is
$$\det \left(\binom{a_i-b_j+1}{j-i+1}\right)_{1\leq i,j\leq n}.$$
\end{lem}
For Lemma \ref{DeterminParti}, the weighted counting of the partitions by area has been studied. A $q$-analogous determinant formula was first obtained by Handa and Mohanty \cite{B.R.Handa},
and later proved by Gessel and Loehr \cite{IraM.Gessel} using an elegant involution argument.

Now we can give a determinant formula for $\overline{f}_m(n)$.
\begin{prop}\label{DeterFormuFM}
Follow the notation in Section 1. The cardinality $\overline{f}_m(n)=\# YT(\mathcal{B}_m^n)$
is the number of lattice paths from $(0,0)$ to $(mn,mn)$ that never go below the path $N(E^mN^m)^{n-1}E^mN^{m-1}$.
Moreover, we have
$$\overline{f}_m(n)=\det \left(\binom{a_i+1}{j-i+1}\right)_{1\leq i,j\leq nm-1},$$
where
$$a_i=hm,\ \ (h-1)m+1\leq i\leq hm,\ 1\leq h\leq n-1$$
and
$$a_{i}=mn,\ \ (n-1)m+1\leq i\leq nm-1.$$
\end{prop}
\begin{proof}
The first part follows by Proposition \ref{TableauxLatticePartit} with respect to $S=\{jm+i\mid 2\leq i\leq m, 0\leq j\leq n-1\}$, where
one can verify that
$$p(T^0)=N^{mn}E^{mn},\ \ \mu(T^0)= \underbrace{(0,0,...,0)}_{nm},$$
and that
$$p(T^M(S))=N(E^mN^m)^{n-1}E^mN^{m-1},$$
$$\mu(T^M(S))= (0,\underbrace{m,...,m}_{m},\underbrace{2m,...,2m}_{m},...,\underbrace{(n-1)m,...,(n-1)m}_{m},
\underbrace{mn,...,mn}_{m-1}).$$

The second part then follows by the first part and Lemma \ref{DeterminParti}, in which we ignore the first element $0$.
\end{proof}

We are interested in finding a nice formula for the generating function of $\overline{f}_m(n)$. Proposition \ref{DeterFormuFM} does not seem to help.
Next we give a recursive formula as follows.
\begin{prop}\label{RecurFormula}
Follow the notation as above. Suppose $m\geq 2$ and $n\geq 1$. Then we have
$$\overline{f}_m(n)=\binom{2mn-1}{mn}-\sum_{i=1}^{n-1}\sum_{t=1}^mg_t(i)\binom{2m(n-i)+m-t-1}{m(n-i)-1}.$$
where,
$$g_t(i)=\binom{2mi-m+t-1}{mi}-\sum_{j=1}^{i-1}\sum_{r=1}^mg_r(j)\binom{2m(i-j)+t-r-1}{m(i-j)-1}$$
and the initial value are
$$g_t(1)=\binom{m+t-1}{m},\ \ 1\leq t\leq m.$$
Note that $g_m(i)=\overline{f}_m(i)$.
\end{prop}
\begin{proof}
We use the lattice path interpretation. The theorem is
an application of the inclusion-exclusion principle. We only illustrate the idea by a simple example, since the
detailed proof is conceptually simple but tedious to present.

Consider the counting of lattice paths for the case $m=3, n=2$ in Figure \ref{YoungT3}.
Firstly, we move $(0,0)$ and $(6,6)$ to $(-1,1)$ (point $M$) and $(5,7)$ (point $F$) respectively. Then we need to count lattice paths
from $M$ to $F$ that never touch the lower red boundary $NE^3N^3E^3N^2$. For this, we subtract from $\binom{6+6}{6}$
the number of bad lattice paths, i.e, those touch the boundary. We classify the bad paths by their first touching of the boundary, at
the points $A$, $B$, $C$ and $D$ in Figure \ref{YoungT3}. This gives
$$\overline{f}_3(2)=\binom{6+6}{6}-1\times \binom{5+6}{5}-1\times \binom{2+5}{2}-\binom{3+1}{3}\binom{2+4}{2}-\binom{3+2}{3}\binom{2+3}{2}=281.$$
The subtracted terms are obtained in a similar way. Let us only explain the last term corresponding to bad paths first touching $D$.
The first term counts the number of ways to go from $M$ to $D$ without touching the boundary.
This corresponds to going from $(0,0)$ to $C$ without going below the boundary, and is hence counted by $g_i(t)$,
but becomes simply $\binom{3+2}{3}$ in the displayed picture; The second term counts the number of ways to go from $D$ to $F$ arbitrarily,
which is clearly $\binom{2+3}{2}$.
\end{proof}

In principle, Proposition \ref{RecurFormula} gives rise to a system of equations on the generating functions $G_t(x)$ of $g_t(n)$ for $t=1,2,\dots,m$. Then one can solve
for $G_t(x)$ for all $t$ and hence $\overline{F}_m(x)$. Indeed, this idea allows us to work out $\overline{F}_m(x)$ for $m=3,4$. Then
in OEIS, we find that $\overline{F}_3(x)$ and $\overline{F}_4(x)$ correspond to [A213403] and [A213404] respectively. We further find that a similar system has been solved in lattice path counting theory, which will be discussed in the next section.

\section{Lattice Path Enumeration}

Before obtaining the main results of this section, we need to introduce some definitions and conclusions about symmetric functions.

\subsection{Symmetric Functions}
We need some basic results on symmetric functions. Most of them can be found in \cite{RP.Stanley99}.

Let $X=(x_1,x_2,...)$ be a set of indeterminates. A \emph{homogeneous symmetric function} of degree $n$ ($n$ is a non-negative integer) over a commutative ring $R$ (with identity) is a formal power series
$$SF(X)=\sum_{\nu} c_{\nu}X^{\nu},$$
where $\nu$ ranges over all weak compositions $\nu=(\nu_1,\nu_2,...)$ of $n$, $c_{\nu}\in R$, $X^{\nu}=x_1^{\nu_1}x_2^{\nu_2}\cdots$,
and $c_\nu=c_{\mu}$ if $\mu$ is obtained from $\nu$ by permuting the entries.

Let $\Lambda^n$ be the set of all homogeneous symmetric functions of degree $n$ over $R$. Then $\Lambda^n$ is a $R$-module. If $R$ is the rational number field $\mathbb{Q}$, then $\Lambda^n$ is a vector space. Let $\Lambda=\Lambda^0\oplus \Lambda^1\oplus \cdots$ be the vector space direct sum.

Let $\mu=(\mu_1, \mu_2,...)_{\ge}$ be a partition of $n$, denoted $\mu\vdash n$. The length of $\mu$, denoted $l(\mu)$, is the number of nonzero entries of $\mu_i$.
There are five important basis of $\Lambda^n$ indexed by partitions, namely:
\begin{enumerate}
  \item The monomial symmetric functions: $m_{\mu}=m_{\mu}(X)=\sum_{\nu}X^{\nu}$, where $\nu$ ranges over all rearrangements of the given partition $\mu$ and $m_{\emptyset}=1$.

  \item The elementary symmetric functions: $e_{\mu}=e_{\mu_1}e_{\mu_2}\cdots$, where $e_{n}=m_{1^n}$ and $e_0=1$.

  \item The complete symmetric functions: $h_{\mu}=h_{\mu_1}h_{\mu_2}\cdots$, where $h_n=\sum_{\kappa \vdash n}m_{\kappa}$ and $h_0=1$.

  \item The power sum symmetric functions: $p_{\mu}=p_{\mu_1}p_{\mu_2}\cdots$, where $p_n=m_{n}$ and $p_0=1$.

  \item The Schur functions: $s_{\mu}=\sum_{\nu\vdash |\mu|}K_{\mu\nu}m_{\nu}$, where $K_{\mu\nu}$ is the Kostka number.
\end{enumerate}

We only use symmetric functions on a finite number of variables, say $x_1,\dots, x_n$. One can treat $x_i=0$ for $i>n$.
Let $\Lambda_n$ be the set of all polynomials $f(X)$ in $\mathbb{Q}[x_1,...,x_n]$ that are invariant under any permutation of the variables. Then $f(X)$ is just a symmetric function on $X=(x_1,...,x_n,0,0,...)$.

The following is the classical definition of Schur functions in the variables $(x_1,...,x_n)$. See, e.g., \cite[Theorem 7.15.1]{RP.Stanley99}.
\begin{lem}[\cite{RP.Stanley99}]\label{SchurDeterm}
Let $\mu=(\mu_1, \mu_2,...)$ be a partition of length $l(\mu)\leq n$. Let $\delta=(n-1,n-2,...,0)$. Then we have
$$s_{\mu}(x_1,...,x_n)=\frac{a_{\mu+\delta}}{a_{\delta}},$$
where $a_{\mu+\delta}=\det (x_i^{\mu_j+n-j})_{i,j=1}^n$ and $a_{\delta}=\det(x_i^{n-j})_{i,j=1}^n$.
\end{lem}
We denote the Vandermonde determinant by
$$\det V(x_1,...,x_n):=(-1)^{\frac{n(n-1)}{2}}a_{\delta}=\prod_{1\leq j<i\leq n}(x_i-x_j).$$

We also use the \emph{plethystic notation} for symmetric functions to simplify the proof. A good reference to plethystic notation is \cite{J.Haglund}. Let $E(x_1,x_2,...)$ be a formal series of rational functions in the parameters $x_1,x_2,...$. We define the plethystic substitution of $E$ into $p_k$, denoted $p_k[E]$, by $p_k[E]=E(x_1^k,x_2^k,...)$. This is to distinguish $p_k[E]$ from the ordinary $k$-th power sum in a set of variables $E$.

\begin{lem}[\cite{J.Haglund}]\label{PlethysticNot}
Let $m\in \mathbb{N}$. We have
\begin{align*}
&e_m[x_1+x_2+\cdots +x_n]=\sum_{i_1< i_2<\cdots <i_m}x_{i_1}x_{i_2}\cdots x_{i_m},
\\&e_m[E-F]=\sum_{i=0}^me_i[E]e_{m-i}[-F],\ \ \ \ e_m[-X]=(-1)^mh_m[X].
\end{align*}
\end{lem}

\begin{lem}[Section 7.6, \cite{RP.Stanley99}]\label{EtoHorHtoE}
Suppose $n\geq 1$. The elementary symmetric functions and the complete symmetric functions have the following relationship:
$$\sum_{i=0}^n(-1)^ie_ih_{n-i}=0.$$
\end{lem}

Now let's give a result involving the Vandermonde determinant.
\begin{lem}\label{W-determ-r}
Let $r\geq 1$ and $r\in \mathbb{N}$. Then
\begin{align*}
\det\left(\begin{array}{ccccc}
1 &w_1 &\cdots & w_1^{\ell-2}& w_1^{-r} \\
1 &w_2 &\cdots & w_2^{\ell-2}& w_2^{-r} \\
\vdots &\vdots &\ddots & \vdots& \vdots \\
1 &w_{\ell} &\cdots & w_{\ell}^{\ell-2}& w_{\ell}^{-r}\\
\end{array}\right)=\frac{(-1)^{(\ell-1)^2}s_{\mu}(w_1,...,w_{\ell})\det V(w_1,...,w_{\ell})}{e_{\ell}^r(w_1,...,w_{\ell})},
\end{align*}
where $\mu=(r-1,r-1,...,r-1)$ with length $\ell(\mu)=\ell-1$.
\end{lem}
\begin{proof}
By direct computation, the left hand side becomes
\begin{align*}
&(w_1^r\cdots w_{\ell}^r)^{-1}
\det\left(\begin{array}{ccccc}
w_1^r &w_1^{r+1} &\cdots & w_1^{\ell+r-2}& 1 \\
w_2^r &w_2^{r+1} &\cdots & w_2^{\ell+r-2}& 1 \\
\vdots &\vdots &\ddots & \vdots& \vdots \\
w_{\ell}^r &w_{\ell}^{r+1} &\cdots & w_{\ell}^{\ell+r-2}& 1 \\
\end{array}\right)
\\=&(-1)^{\frac{(\ell-1)(\ell-2)}{2}}(w_1^r\cdots w_{\ell}^r)^{-1}a_{(\ell+r-2,\ell+r-3,...,r,0)}(w_1,...,w_{\ell})
\\(\text{by Lemma \ref{SchurDeterm}})\quad =&(-1)^{\frac{(\ell-1)(\ell-2)}{2}}(w_1^r\cdots w_{\ell}^r)^{-1}s_{\mu}(w_1,...,w_{\ell})\cdot a_{(\ell-1,\ell-2,...,1,0)}(w_1,...,w_{\ell})
\\=&\frac{(-1)^{(\ell-1)^2}s_{\mu}(w_1,...,w_{\ell})\det V(w_1,...,w_{\ell})}{e_{\ell}^r(w_1,...,w_{\ell})}.
\end{align*}
This completes the proof.
\end{proof}

\begin{lem}\label{HESlambda}
Let $l\geq 1$ and $i\geq 0$. Let $\lambda=(i,i,...,i)$ with length $\ell(\lambda)=\ell-1$. We have
$$h_i(w_1^{-1},w_2^{-1},...,w_{\ell}^{-1})\cdot e_{\ell}^i(w_1,w_2,...,w_{\ell})=s_{\lambda}(w_1,w_2,...,w_{\ell}).$$
\end{lem}
\begin{proof}
Let $\delta=(\ell-1,\ell-2,...,0)$.
By the classical definition of Schur functions in Lemma \ref{SchurDeterm} with respect to $\lambda=(i,i,...,i,0)$, we have
\begin{align*}
&s_{\lambda}(w_1,w_2,...,w_{\ell})=\frac{a_{\lambda+\delta}(w_1,w_2,...,w_{\ell})}{a_{\delta}(w_1,w_2,...,w_{\ell})}
\\=&\det\left(\begin{array}{ccccc}
w_1^{i+\ell-1} &w_1^{i+\ell-2} &\cdots & w_1^{i+1}& 1 \\
w_2^{i+\ell-1} &w_2^{i+\ell-2} &\cdots & w_2^{i+1}& 1 \\
\vdots &\vdots &\ddots & \vdots& \vdots \\
w_{\ell}^{i+\ell-1} &w_{\ell}^{i+\ell-2} &\cdots & w_{\ell}^{i+1}& 1 \\
\end{array}\right)
\left(\det\left(\begin{array}{ccccc}
w_1^{\ell-1} &w_1^{\ell-2} &\cdots & w_1& 1 \\
w_2^{\ell-1} &w_2^{\ell-2} &\cdots & w_2 & 1 \\
\vdots &\vdots &\ddots & \vdots& \vdots \\
w_{\ell}^{\ell-1} &w_{\ell}^{\ell-2} &\cdots & w_{\ell}& 1 \\
\end{array}\right)\right)^{-1}
\\=&\frac{(w_1w_2\cdots w_{\ell})^{i+\ell-1}}{(w_1w_2\cdots w_{\ell})^{\ell-1}}
\det\left(\begin{array}{cccc}
1  &\cdots & w_1^{2-\ell}& w_1^{-i-\ell+1} \\
1  &\cdots & w_2^{2-\ell}& w_2^{-i-\ell+1} \\
\vdots  &\ddots & \vdots& \vdots \\
1  &\cdots & w_{\ell}^{2-\ell}& w_{\ell}^{-i-\ell+1} \\
\end{array}\right)
\left(\det\left(\begin{array}{cccc}
1  &\cdots & w_1^{2-\ell}& w_1^{1-\ell} \\
1  &\cdots & w_2^{2-\ell}& w_2^{1-\ell} \\
\vdots  &\ddots & \vdots& \vdots \\
1  &\cdots & w_{\ell}^{2-\ell}& w_{\ell}^{1-\ell} \\
\end{array}\right)\right)^{-1}
\\=&(w_1w_2\cdots w_{\ell})^i
\det\left(\begin{array}{cccc}
(w_1^{-1})^{i+\ell-1}  &\cdots & w_1^{-\ell}& 1 \\
(w_2^{-1})^{i+\ell-1}  &\cdots & w_2^{-\ell}& 1 \\
\vdots  &\ddots & \vdots& \vdots \\
(w_{\ell}^{-1})^{i+\ell-1}  &\cdots & w_{\ell}^{-\ell}& 1 \\
\end{array}\right)
\left(\det\left(\begin{array}{cccc}
(w_1^{-1})^{\ell-1}  &\cdots & w_1^{-\ell}& 1 \\
(w_2^{-1})^{\ell-1}  &\cdots & w_2^{-\ell}& 1 \\
\vdots  &\ddots & \vdots& \vdots \\
(w_{\ell}^{-1})^{\ell-1}  &\cdots & w_{\ell}^{-\ell}& 1 \\
\end{array}\right)\right)^{-1}
\\=&(w_1w_2\cdots w_{\ell})^i\cdot s_{i}(w_1^{-1},w_2^{-1},...,w_{\ell}^{-1})
\\=&e_{\ell}^{i}(w_1,w_2,...,w_{\ell})h_i(w_1^{-1},w_2^{-1},...,w_{\ell}^{-1}).
\end{align*}
The last ``$=$" is due to the fact that $s_i=h_i$.
\end{proof}

We will use several variations of the well-known Vandermonde's identity
$$\sum_{i=0}^{n}\binom{a}{i}\binom{b}{n-i}=\binom{a+b}{n}.$$
Vandermonde's identity can be easily proved by equating coefficients of $x^n$ on both sides of the equation $(1+x)^a\cdot (1+x)^b=(1+x)^{a+b}$.
This trick is sufficient for our purpose.

\begin{lem}\label{CTbinomTT1}
Let $j\leq \ell$ and $j+1\leq r$. We have
$$\sum_{m=0}^{r-1-j}(-1)^{j+r+m+1}\binom{\ell-j}{r-1-m-j}=(-1)^{r-j-1}\binom{\ell-j-1}{r-j-1}.$$
\end{lem}
\begin{lem}\label{CTbinomTT3}
Let $j\leq \ell$, $i+j\leq r-1$ and $\ell\geq 0$. We have
$$\sum_{m=i}^{r-1-j}(-1)^{r+m+i+j}\binom{\ell-j}{r-1-m-j}\binom{m+l}{i+\ell}=(-1)^{i-1}\binom{r-1}{i+j}.$$
\end{lem}
\begin{lem}\label{CTbinomTT2}
Let $j\leq \ell$, $j+1\leq r$ and $\ell\geq 0$. We have
$$\sum_{m=0}^{r-1-j}(-1)^{j+r+m}\binom{\ell-j}{r-1-m-j}\binom{m+\ell}{\ell}=-\binom{r-1}{j}.$$
\end{lem}
\begin{proof}[Sketched proofs of Lemmas \ref{CTbinomTT1}, \ref{CTbinomTT3} and \ref{CTbinomTT2}]
Lemma \ref{CTbinomTT1} follows by equating the coefficients of $x^{r-j-1}$ on
both sides of the identity
\begin{align*}
\sum_{k=0}^{\ell-j}\binom{\ell-j}{k}x^k \cdot \sum_{k\geq 0}(-x)^k =\frac{(1+x)^{\ell-j}}{(1+x)} =(1+x)^{\ell-j-1}=\sum_{k=0}^{\ell-j-1}\binom{\ell-j-1}{k}x^k.
\end{align*}
Lemma \ref{CTbinomTT3} follows by equating the coefficients of $x^{r-1-i-j}$ on
both sides of the identity
\begin{small}
\begin{align*}
\sum_{k\ge 0} \binom{i+\ell+k}{i+j} (-x)^k \cdot \sum_{k=0}^{\ell-j} \binom{\ell-j}{k} (-x)^k =\frac{(1+x)^{\ell-j}}{ (1+x)^{i+\ell+1}} =\frac{1}{(1+x)^{i+j+1}}=\sum_{k\ge 0} \binom{i+j+k}{i+\ell}(-x)^k.
\end{align*}
\end{small}
Lemma \ref{CTbinomTT2} follows in a similar way, and we omit the proof.
\end{proof}

\subsection{Lattice Paths}\label{LatticePath41}

In this subsection, we focus on the counting of lattice paths from $(0,0)$ to $(m,n)$ that \textsf{never go above} a given path $P$ by using steps $N=(0,1)$ and $E=(1,0)$. For example, let's flip the paths discussed in Section \ref{YTWallsToPath}. Then $\overline{f}_m(n)$ counts the number of lattice paths from $(0,0)$ to $(mn-1,mn)$ that never go above the path $(N^mE^m)^{n-1}N^{m}E^{m-1}$. We use the convention $\overline{f}_m(0)=1$.
When $m=1$, $\sum_{n\geq 0}\overline{f}_1(n)x^n=C(x)$ is the Catalan generating function.

Bonin, de Mier and Noy (\cite{JBonin03}) made a connection between lattice paths and matroids. They defined the \emph{Tutte polynomial} of a lattice path
$P$ as follows.
$$t(P;z,y)=\sum_{\sigma: \text{not go above }P}z^{i(\sigma)}y^{e(\sigma)},$$
where the sum ranges over all lattice paths $\sigma$ that never go above $P$, $i(\sigma)$ is the number of common $N$ steps of $\sigma$ and $P$, and
$e(\sigma)$ is the number of common $E$ steps of $\sigma$ and $P$ before the first $N$ step.
In particular, $t(P;1,1)$ is the number of all paths that never go above $P$.

We need the following result, which was first obtained from Matroid Theory.
\begin{lem}[\cite{JBonin03}]\label{TuttePloy-NE}
Let $PN$ be the path obtained from $P$ by appending a $N$ step, and $PE$ be the path obtained similarly. Then we have
\begin{align*}
&t(PN;z,y)=z\cdot t(P;z,y);
\\& t(PE;z,y)=\frac{z}{z-1}\cdot t(P;z,y)+\left(y-\frac{z}{z-1}\right)\cdot t(P;1,y).
\end{align*}
\end{lem}
Readers are invited to find a combinatorial proof of the above lemma.

Let $q(n)$ be the number of lattice paths from $(0,0)$ to $(\ell n,kn)$ that never go above the path $(N^kE^{\ell})^n$, where
 $k,\ell\in\mathbb{P}$ and $n\in \mathbb{N}$. Using Lemma \ref{TuttePloy-NE}, de Mier and Noy obtained a nice expression of the generating function
$Q(x)=\sum_{n\geq 0}q(n)x^n$. See \cite[Theorem 1]{AnnadeMier} or Corollary \ref{Annader-Mie} below. For completeness, we briefly describe their proof process as follows.

Let $P_n=(N^kE^{\ell})^n$ and $A_n=A_n(z,y)=t(P_n;z,y)$, where $A_0=1$. Define the operator $\Phi$ by
$$\Phi A(z,y)=\frac{z}{z-1}A(z,y)+\left(y-\frac{z}{z-1}\right)A(1,y).$$
Then Lemma \ref{TuttePloy-NE} gives $A_{n+1}=\Phi^{\ell}(z^kA_n)$.
For each $n\in \mathbb{N}$ and $1\leq i\leq \ell$, define the polynomials
$$B_{i,n}=\Phi^i\left(z^kA_n(z,y)\right),\ \ C_{i,n}=B_{i,n}(1,y),$$
where we set $C_{0,n}(y)=A_n(1,y)$.
Thus we have $B_{\ell,n}=A_{n+1}$ and $C_{0,n}(1)=A_n(1,1)$.
Denote their generating functions by
$$A=A(x)=\sum_{n\geq 0}A_nx^n,\ \ \ C_i=C_i(x)=\sum_{n\geq 0}C_{i,n}x^n,\ \ \ 0\leq i\leq \ell.$$

Based on the above definitions, we have
\begin{align}
&B_{1,n}=\frac{z}{z-1}z^kA_n+\left(y-\frac{z}{z-1}\right)C_{0,n},\nonumber
\\&B_{2,n}=\frac{z}{z-1}B_{1,n}+\left(y-\frac{z}{z-1}\right)C_{1,n},\nonumber
\\&...\label{B1B2B3C0C1C2}
\\&B_{\ell,n}=\frac{z}{z-1}B_{\ell-1,n}+\left(y-\frac{z}{z-1}\right)C_{\ell-1,n},\nonumber
\\&A_{n+1}=B_{\ell,n}.\nonumber
\end{align}

In the above equations, we start from the last equation and repeatedly replace $B_{i,n}$ from the previous equation. We have
$$\sum_{n\geq 0} A_{n+1}x^n=\frac{A-1}{x}=\frac{z^{k+\ell}}{(z-1)^{\ell}}A+(yz-y-z)\sum_{i=1}^{\ell}\frac{z^{i-1}}{(z-1)^i}C_{\ell-i}.$$
Simplifying and setting $y=1$ give
\begin{equation}\label{w1w2w3-wl}
A\left((z-1)^{\ell}-xz^{k+\ell}\right)=(z-1)^{\ell}-x\sum_{i=1}^{\ell}z^{i+1}(z-1)^{\ell-i}C_{\ell-i}.
\end{equation}
This functional equation can be solved by the well-known kernel method.

The factor $\left((z-1)^{\ell}-xz^{k+\ell}\right)$ on the left hand side is called the kennel.
As a polynomial in $z$, it has $k+\ell$ roots. By Puiseus's theorem, \cite[Chapter 6]{RP.Stanley99}, these roots
can be treated as elements in the field of fractional Laurent series
$$\mathbb{C}^{\text{fra}}((x))=\Big\{\sum_{n\geq N}a_{n,M}x^{n/M}: a_{n,M}\in \mathbb{C},  N\in\mathbb{Z},\ M\in \mathbb{P}   \Big\}.$$
Moreover, by \cite[Proposition 6.1.8]{RP.Stanley99}, exactly $\ell$ roots of the kernel are fractional power series.
Denote them by $w_1(x),w_2(x)$,$...,w_{\ell}(x)$. Then $w_i(0)=1$ for $1\leq i \leq \ell$.

The substitution $z=w_j$ in Equation \eqref{w1w2w3-wl} is valid for $1\leq j\leq \ell$. This gives
a system of $\ell$ linear equations in $C_j$:
\begin{equation}\label{system-equat}
\sum_{i=1}^{\ell}w_j^{i-1}(w_j-1)^{\ell-i}xC_{\ell-i}=(w_j-1)^{\ell},\ \ 1\leq j\leq \ell.
\end{equation}
Because $Q(x)=C_0=\sum_{n\geq 0}A_n(1,1)x^n$, de Mier and Noy only obtained a nice expression for $C_0$ by using Lagrange's interpolation formula and an identity in \cite[Exercise 7.4]{RP.Stanley99}. We find nice expressions of $C_r$ for all $r$.

Our intermediate result is the following.
\begin{thm}
Let $\ell\in\mathbb{P}$ and $1\leq r\leq \ell$.
Let $\overline{w}_i=\frac{w_i-1}{w_i}$ for $1\leq i\leq \ell$. Then we have
\begin{align}\label{xC1-rw22}
xC_{\ell-r}=\sum_{m=0}^{r-1}e_{r-1-m}(\overline{w_1},\cdots,\overline{w_\ell})
\left(\sum_{i=0}^m(-1)^{r+m+i}\binom{m+\ell}{i+\ell}
\frac{s_{\mu}(w_1,...,w_{\ell})}{e_{\ell}^{i}}+(-1)^{r+m+1}e_{\ell}\right),
\end{align}
where $e_{\ell}=e_\ell(w_1,...,w_{\ell})$ and $\mu=(i,i,...,i)$ with length $\ell(\mu)=\ell-1$.
\end{thm}
\begin{proof}
By the system \eqref{system-equat}, we have
\begin{equation}\label{ref-system-equat}
\sum_{i=0}^{\ell-1}\left(\frac{w_i}{w_j-1}\right)^ixC_{\ell-i-1}=w_j-1,\ \ 1\leq j\leq \ell.
\end{equation}
The left side of equations \eqref{ref-system-equat} can be regarded as the result of evaluating the polynomial
$T(X)=\sum_{i=0}^{\ell-1}(xC_{\ell-i-1})X^i$ of degree $\ell-1$ at $X=X_j=\frac{w_j}{w_j-1}$ for $1\leq j\leq \ell$.
That is, $T(X_j)=w_j-1$ for $1\leq j\leq \ell$.
Then using Lagrange's interpolation formulas, we have
$$T(X)=\sum_{j=1}^{\ell}\prod_{i\neq j}\frac{X-X_i}{X_j-X_i}T(X_j).$$
This is the same as in \cite{AnnadeMier} up to here.

Now consider the coefficient of $X^{r-1}$ in $T(X)$. We obtain
\begin{align*}
xC_{\ell-r}&=\sum_{j=1}^{\ell}\prod_{i\neq j}\frac{1}{X_j-X_i}\left((-1)^{\ell-r}\sum_{1\leq k_1<k_2<\cdots <k_{r-1}\leq \ell\atop k_1,...,k_{r-1}\neq j}\frac{X_1X_2\cdots X_{\ell}}{X_{k_1}\cdots X_{k_{r-1}}X_j}\right)T(X_j)
\\&=(-1)^{\ell-r}\frac{w_1w_2\cdots w_{\ell}}{(w_1-1)\cdots (w_{\ell}-1)}\sum_{j=1}^{\ell}\prod_{i\neq j}\frac{(w_j-1)(w_i-1)}{w_i-w_j}\cdot
\\&\ \ \cdot\left(\sum_{1\leq k_1<k_2<\cdots <k_{r-1}\leq \ell\atop k_1,...,k_{r-1}\neq j}\frac{(w_{k_1}-1)\cdots (w_{k_{r-1}}-1)(w_j-1)}{w_{k_1}w_{k_2}\cdots w_{k_{r-1}}w_j}\right)(w_j-1)
\\&=(-1)^{\ell-r}w_1\cdots w_{\ell}\sum_{j=1}^{\ell}\prod_{i\neq j}\frac{-1}{w_j-w_i}\cdot\left(\sum_{1\leq k_1<\cdots <k_{r-1}\leq \ell\atop k_1,...,k_{r-1}\neq j}\frac{(w_{k_1}-1)\cdots (w_{k_{r-1}}-1)}{w_{k_1}w_{k_2}\cdots w_{k_{r-1}}w_j}\right)(w_j-1)^{\ell}
\\&=(-1)^{r+1}w_1\cdots w_{\ell}\sum_{j=1}^{\ell}\prod_{i\neq j}\frac{1}{w_j-w_i}\cdot\left(\sum_{1\leq k_1<\cdots <k_{r-1}\leq \ell\atop k_1,...,k_{r-1}\neq j}\frac{(w_{k_1}-1)\cdots (w_{k_{r-1}}-1)}{w_{k_1}w_{k_2}\cdots w_{k_{r-1}}}\right)\frac{(w_j-1)^{\ell}}{w_j}.
\end{align*}
Denote by
$$g(w_j)=\left(\sum_{1\leq k_1<\cdots <k_{r-1}\leq \ell\atop k_1,...,k_{r-1}\neq j}\frac{(w_{k_1}-1)\cdots (w_{k_{r-1}}-1)}{w_{k_1}w_{k_2}\cdots w_{k_{r-1}}}\right)\frac{(w_j-1)^{\ell}}{w_j}.$$
Then by definition of the Vandermonde determinant, we have
\begin{align*}
xC_{\ell-r}&=(-1)^{r+1}w_1\cdots w_{\ell}\sum_{j=1}^{\ell}\prod_{i\neq j}\frac{1}{w_j-w_i}\cdot g(w_j)
\\&=(-1)^{r+1}w_1\cdots w_{\ell}\sum_{j=1}^{\ell}\frac{(-1)^{\ell-j}g(w_j)\det V(w_1,...,w_{\ell}\setminus w_j)}{\det V(w_1,...,w_{\ell})},
\end{align*}
where $\det V(w_1,...,w_{\ell}\setminus w_j)$ represents the Vandermonde determinant of the variables\\
 $w_1,...,w_{j-1}, w_{j+1},...,w_{\ell}$.
Furthermore, we have
\begin{align*}
G=\sum_{j=1}^{\ell}(-1)^{\ell+j}g(w_j)\det V(w_1,...,w_{\ell}\setminus w_j)
=\det\left(\begin{array}{ccccc}
1 &w_1 &\cdots & w_1^{\ell-2}& g(w_1) \\
1 &w_2 &\cdots & w_2^{\ell-2}& g(w_1) \\
1 &w_3 &\cdots & w_3^{\ell-2}& g(w_3) \\
\vdots &\vdots &\ddots & \vdots& \vdots \\
1 &w_{\ell} &\cdots & w_{\ell}^{\ell-2}& g(w_{\ell})\\
\end{array}\right).
 \end{align*}
Observe that $G$ becomes i) $0$ if $g(w)$ is $1,w,\dots, w^{\ell-2}$; ii) $\det V(w_1,...,w_{\ell})$ if $g(w)=w^{\ell-1}$; iii)
$\frac{(-1)^{(\ell-1)^2}s_{\mu}\det V(w_1,...,w_{\ell})}{e_{\ell}^{i+1}(w_1,...,w_{\ell})}$ (by Lemma \ref{W-determ-r}) if $g(w)=w^{-i-1}$ for $i\geq 0$.
Thus we shall write $g(w_j)$ as a polynomial in $w_j$ with coefficients symmetric in $w_1,\dots, w_\ell$.

To this end, we use plethestic notation by treating $\overline{w}_i=\frac{w_i-1}{w_i}$ for $1\leq i\leq \ell$ as variables.
Then $w_i=\frac{\overline{w}_i-1}{\overline{w}_i}$ is not a variable. Let $\overline{w}=\overline{w_1}+\overline{w_2}+\cdots +\overline{w_\ell}$.
By Lemma \ref{PlethysticNot}, we have
\begin{align*}
g(w_j)&=e_{r-1}[\overline{w}-\overline{w}_j]\cdot \frac{(w_j-1)^{\ell}}{w_j}
=\sum_{m=0}^{r-1}e_{r-1-m}[\overline{w}]\cdot e_{m}[-\overline{w_j}]\cdot \frac{(w_j-1)^{\ell}}{w_j}
\\&=\sum_{m=0}^{r-1}e_{r-1-m}[\overline{w}]\cdot (-1)^m\cdot \frac{(w_j-1)^{m+\ell}}{w_j^{m+1}}
=\sum_{m=0}^{r-1}(-1)^m e_{r-1-m}[\overline{w}]\sum_{i=0}^{m+\ell}(-1)^i\binom{m+\ell}{i}w_j^{\ell-i-1}.
\end{align*}
Therefore, we obtain
\begin{align*}
G&=\sum_{m=0}^{r-1}(-1)^m e_{r-1-m}[\overline{w}]\left(\sum_{i=0}^m(-1)^{\ell+i+(\ell-1)^2}\binom{m+\ell}{i+\ell}\frac{s_{\mu}\det V(w_1,...,w_{\ell})}{e_{\ell}^{i+1}}+\det V(w_1,...,w_{\ell})\right).
\end{align*}

Now we have
\begin{align*}
xC_{\ell-r}&=(-1)^{r+1}w_1\cdots w_{\ell}\left(\sum_{m=0}^{r-1}(-1)^m e_{r-1-m}[\overline{w}]\left(\sum_{i=0}^m(-1)^{i+1}\binom{m+\ell}{i+\ell}
\frac{s_{\mu}}{e_{\ell}^{i+1}}+1\right)\right)
\\&=\sum_{m=0}^{r-1}e_{r-1-m}[\overline{w}]\left(\sum_{i=0}^m(-1)^{r+m+i}\binom{m+\ell}{i+\ell}
\frac{s_{\mu}}{e_{\ell}^{i}}+(-1)^{r+m+1}e_{\ell}\right).
\end{align*}
This completes the proof.
\end{proof}

Now we present our proof of Theorem \ref{MainResinIntrou}.
\begin{proof}[Proof of Theorem \ref{MainResinIntrou}]
For $1\leq r\leq \ell$, let $f_r(n)$ be the number of lattice paths from $(0,0)$ to $(\ell n-r,kn)$ that never go above the path $(N^kE^{\ell})^{n-1}N^kE^{\ell-r}$. We consider the generating function $F_r(x)=\sum_{n\geq 0}f_r(n)x^n$ with the convention $f_r(0)=1$. Then it is easy to see that
$F_r(x)=xC_{\ell-r}+1$ for $1\leq r\leq \ell$.

To simplify Equation \eqref{xC1-rw22}, we need to write $e_{r-1-m}[\overline{w}]$ in terms of $e_i=e_i(w_1,\dots, w_\ell)$.
Consider the generating function of $e_i[\overline{w}]$ as follows.
\begin{align*}
\prod_{i=1}^{\ell}(1+t\overline{w}_i)&=\prod_{i=1}^{\ell}\left(1+t(1-w_i^{-1})\right)
=\prod_{i=1}^{\ell}(1+t-t\cdot w_i^{-1})
\\&=(1+t)^{\ell}\prod_{i=1}^{\ell}\left(1-\frac{t}{1+t}w_i^{-1}\right)
\\&=(1+t)^{\ell}\sum_{j\geq 0}e_{j}(w_1^{-1},w_2^{-1},...w_{\ell}^{-1})(-1)^j\left(\frac{t}{1+t}\right)^j
\\&=\sum_{j\geq 0}e_{j}(w_1^{-1},w_2^{-1},...,w_{\ell}^{-1})(-1)^jt^j(1+t)^{\ell-j}.
\end{align*}
Taking the coefficient of $t^{r-m-1}$ on both sides of the above equation, we obtain
\begin{equation}\label{ER1Mer1m}
\begin{aligned}
e_{r-1-m}[\overline{w}]&=[t^{r-m-1-j}]\sum_{j=0}^{r-m-1}(-1)^j
e_j(w_1^{-1},w_2^{-1},...,w_{\ell}^{-1})(1+t)^{\ell-j}
\\&=\sum_{j=0}^{r-m-1}(-1)^j\binom{\ell-j}{r-m-1-j}e_j(w_1^{-1},w_2^{-1},...,w_{\ell}^{-1})
\\&= e_\ell^{-1}\sum_{j=0}^{r-m-1}(-1)^j\binom{\ell-j}{r-m-1-j}e_{\ell-j}(w_1,w_2,...,w_{\ell}).
\end{aligned}
\end{equation}
By Equation \eqref{xC1-rw22}, we first obtain
\begin{align}\label{FSLast2}
\sum_{m=0}^{r-1}e_{r-1-m}[\overline{w}](-1)^{r+m+1}e_{\ell}
=&\sum_{m=0}^{r-1}\sum_{j=0}^{r-1-m}(-1)^{r+m+1+j}
\binom{\ell-j}{r-1-m-j}e_{\ell-j}(w_1,w_2,...,w_{\ell}) \nonumber
\\=&\sum_{j=0}^{r-1}\sum_{m=0}^{r-1-j}(-1)^{j+r+m+1}
\binom{\ell-j}{r-1-m-j}e_{\ell-j}(w_1,w_2,...,w_{\ell}) \nonumber
\\(\text{By Lemma }\ref{CTbinomTT1}.)\ \ \ =&\sum_{j=0}^{r-1}(-1)^{r-j-1}\binom{\ell-j-1}{r-j-1}e_{\ell-j}(w_1,w_2,...,w_{\ell})\nonumber
\\=&\sum_{i=0}^{r-1}(-1)^i\binom{\ell-r+i}{i}e_{\ell-r+1+i}(w_1,w_2,...,w_{\ell}).
\end{align}

By $F_r(x)=xC_{\ell-r}+1$ and observing Equation \eqref{xC1-rw22}, we assert that
\begin{align}\label{SUM=0}
1+\sum_{m=0}^{r-1}e_{r-1-m}[\overline{w}]\left(\sum_{i=0}^m(-1)^{r+m+i}\binom{m+\ell}{i+\ell}
\frac{s_{\mu}}{e_{\ell}^{i}}\right)=0.
\end{align}
i.e.,
\begin{align*}
1+\sum_{m=0}^{r-1}\left(\sum_{j=0}^{r-m-1}(-1)^j\binom{\ell-j}{r-m-1-j}
e_j(w_1^{-1},...,w_{\ell}^{-1})\right)\left(\sum_{i=0}^m(-1)^{r+m+i}\binom{m+\ell}{i+\ell}
\frac{s_{\mu}}{e_{\ell}^{i}}\right)=0.
\end{align*}
We first consider the case $i=0$ on the left side of the above equation. We have
\begin{align*}
&1+\sum_{m=0}^{r-1}\left(\sum_{j=0}^{r-m-1}(-1)^j\binom{\ell-j}{r-m-1-j}
e_j(w_1^{-1},...,w_{\ell}^{-1})\right)\left((-1)^{r+m}\binom{m+\ell}{\ell}\right)
\\=&1+\sum_{j=0}^{r-1}\sum_{m=0}^{r-1-j}(-1)^{r+m+j}\binom{\ell-j}{r-1-m-j}\binom{m+\ell}{\ell}
e_j(w_1^{-1},...,w_{\ell}^{-1})
\\=&1+\sum_{j=0}^{r-1}-\binom{r-1}{j}e_j(w_1^{-1},...,w_{\ell}^{-1})\ \ \ (\text{By Lemma }\ref{CTbinomTT2}.)
\\=&-\sum_{j=1}^{r-1}\binom{r-1}{j}e_j(w_1^{-1},...,w_{\ell}^{-1}).
\end{align*}

Now we consider the case $i>0$. By Equation \eqref{SUM=0}, we have
\begin{align*}
&\sum_{m=1}^{r-1}\sum_{j=0}^{r-m-1}\sum_{i=1}^m(-1)^{r+m+i+j}
\binom{\ell-j}{r-m-1-j}\binom{m+\ell}{i+\ell}e_j(w_1^{-1},...,w_{\ell}^{-1})\frac{s_{\mu}}{e_{\ell}^{i}}
\\=&\sum_{i=1}^{r-1}\sum_{j=0}^{r-1-i}\sum_{m=i}^{r-1-j}(-1)^{r+m+i+j}
\binom{\ell-j}{r-m-1-j}\binom{m+\ell}{i+\ell}e_j(w_1^{-1},...,w_{\ell}^{-1})\frac{s_{\mu}}{e_{\ell}^{i}}
\\=&\sum_{i=1}^{r-1}\sum_{j=0}^{r-1-i}(-1)^{i-1}
\binom{r-1}{j+i}e_j(w_1^{-1},...,w_{\ell}^{-1})\frac{s_{\mu}}{e_{\ell}^{i}}.\ \ \ (\text{By Lemma } \ref{CTbinomTT3}.)
\\=&\sum_{i=1}^{r-1}\sum_{j=0}^{r-1-i}(-1)^{i-1}
\binom{r-1}{j+i}e_j(w_1^{-1},...,w_{\ell}^{-1})h_i(w_1^{-1},...,w_{\ell}^{-1})\ \ \ (\text{By Lemma } \ref{HESlambda}.)
\\=&\binom{r-1}{1}e_1(w_1^{-1},...,w_{\ell}^{-1})+\cdots
+\binom{r-1}{r-1}e_{r-1}(w_1^{-1},...,w_{\ell}^{-1})\ \ \ (\text{By Lemma } \ref{EtoHorHtoE}.)
\\=&\sum_{j=1}^{r-1}\binom{r-1}{j}e_j(w_1^{-1},...,w_{\ell}^{-1}).
\end{align*}
Therefore, Equation \eqref{SUM=0} is correct. Furthermore, combining Equations \eqref{xC1-rw22}, \eqref{FSLast2}, and \eqref{SUM=0}, we have
$$F_r(x)=xC_{\ell-r}+1=\sum_{i=0}^{r-1}(-1)^i\binom{\ell-r+i}{i}e_{\ell-r+1+i}(w_1,w_2,...,w_{\ell}).$$
This completes the proof.
\end{proof}

The following results were first obtained by de Mier and Noy.
\begin{cor}[\cite{AnnadeMier}]\label{Annader-Mie}
Following the notation in Theorem \ref{MainResinIntrou}. Let $q(n)$ be the number of lattice paths from $(0,0)$ to $(\ell n,kn)$ that never go above the path $(N^kE^{\ell})^n$. Then the generating function $Q(x)=\sum_{n\geq 0}q(n)x^n$ is given by
$$Q(x)=\frac{-1}{x}(1-w_1)\cdots (1-w_{\ell}).$$
\end{cor}
\begin{proof}
By the lattice path interpretation of $C_0$, we have $Q(x)=C_0$. By Theorem \ref{MainResinIntrou}, we have
\begin{align*}
Q(x)&=\frac{1}{x}(F_{\ell}(x)-1)=\frac{1}{x}\left(\sum_{i=0}^{\ell-1}(-1)^ie_{i+1}-1\right)
\\&=\frac{1}{x}\left(\sum_{i=0}^{\ell}(-1)^{i-1}e_{i}\right)=\frac{-1}{x}(1-w_1)(1-w_2)\cdots (1-w_{\ell}).
\end{align*}
This completes the proof.
\end{proof}

\begin{cor}\label{CaseF1Thm}
Following the notation in Theorem \ref{MainResinIntrou}. If $r=1$, then we have
$$F_1(x)=w_1w_2\cdots w_{\ell}.$$
If $r=2$, then we have
$$F_2(x)=\left(1-\ell+\sum_{j=1}^{\ell}\frac{1}{w_j}\right)\cdot\prod_{j=1}^{\ell}w_j.$$
In particular, if $r=1$ and $k=\ell$, then
$$F_1(x)=\prod_{j=1}^{\ell}C(\xi^jx^{\frac{1}{\ell}}).$$
If $r=2$ and $k=\ell$, then
$$F_2(x)=\left(1-\ell+\sum_{j=1}^{\ell}\frac{1}{C(\xi^jx^{\frac{1}{\ell}})}\right)
\cdot\prod_{j=1}^{\ell}C(\xi^jx^{\frac{1}{\ell}}),$$
where $C(x)$ is the Catalan function and $\xi$ is a primitive $\ell$-th root of unity $\xi=e^{\frac{2\pi i}{\ell}}$.
\end{cor}

\section{Exponential Formulas for $Q(x)$ and $F_1(x)$}
In this section, we mainly give two exponential formulas in Theorems \ref{QQXXinPhd} and \ref{Mainresult22},
using ideas from the doctoral thesis of the second author, especially \cite[Chapter 1, 1-5]{XinPhd}.
The thesis also includes some basic knowledge of Puiseux's theorem.

Throughout this section, we always use the following factorization
\begin{align}\label{ExponentialW-11x}
  (w-1)^{\ell}-xw^{k+\ell}=-x (w-w_1)\cdots (w-w_\ell) (w-w_{\ell+1})\cdots (w-w_{k+\ell}).
\end{align}
By Puiseux's theorem, we may assume $w_i$ are all fractional Laurent series in $\mathbb{C}^{\text{fra}}((x))$ and
$w_1,...,w_{\ell}$ is the unique fractional power series. Moreover, $w_i|_{x=0}=1$ for $1\leq i \leq \ell$.

Now we state our first result.
\begin{thm}\label{QQXXinPhd}
Follow the notation in Corollary \ref{Annader-Mie}. Let $q(n)$ be the number of lattice paths from $(0,0)$ to $(\ell n,kn)$ that never go above the path $(N^kE^{\ell})^n$. Then we have
$$Q(x)=\sum_{n\geq 0}q(n)x^n=\exp \left(\sum_{n\geq 1}\binom{kn+\ell n}{\ell n}\frac{x^n}{n}\right).$$
\end{thm}
\begin{proof}
Put $w=u+1$ in \eqref{ExponentialW-11x}. We obtain the factorization
\begin{align}\label{ExponentialUU-22x}
u^{\ell}-x(u+1)^{k+\ell}=-x(u-u_1)(u-u_2)\cdots (u-u_{k+\ell}),
\end{align}
where $u_i=w_i-1, \ 1\leq i \leq k+\ell$ are fractional Laurent series,
and $u_1,...,u_{\ell}$ are the only fractional power series. Moreover, $u_i|_{x=0}=0$ for $1\leq i \leq \ell$.

Rewrite \eqref{ExponentialUU-22x} as
\begin{align}\label{WtoUAX}
1-\frac{x(u+1)^{k+\ell}}{u^{\ell}}=A(x)\Big(1-\frac{u_1}{u}\Big)\cdots\Big(1-\frac{u_{\ell}}{u}\Big)
\Big(1-\frac{u}{u_{\ell+1}}\Big)\cdots \Big(1-\frac{u}{u_{k+\ell}}\Big).
\end{align}
Then by comparing the lowest power terms of $u$ on both sides of the above equation, we obtain
$$-\frac{x}{u^{\ell}}=A(x)\cdot (-\frac{u_1}{u})\cdots (-\frac{u_{\ell}}{u}).$$
Therefore by Corollary \ref{Annader-Mie}, we have
$$(A(x))^{-1}=\frac{(-1)^{\ell+1}}{x}u_1\cdots u_{\ell}=\frac{(-1)^{\ell+1}}{x}(w_1-1)\cdots (w_{\ell}-1)=Q(x).$$

Now we extract $(A(x))^{-1}$ through taking logarithms and working in
$\mathbb{C}((u))^{\text{fra}}((x))$, i.e., the field of fractional Laurent series in $x$ with coefficients Laurent series in $u$.
We have
\begin{align}
\ln\left(1-\frac{x(u+1)^{k+\ell}}{u^{\ell}}\right)^{-1}&=\ln (A(x))^{-1}+\sum_{i=1}^{\ell}\ln\left(1-\frac{u_i}{u}\right)^{-1}
+\sum_{j=\ell+1}^{\ell+k}\ln\left(1-\frac{u}{u_i}\right)^{-1}\nonumber
\\&=\ln (A(x))^{-1}+\sum_{i=1}^{\ell}\sum_{n\geq 1}\frac{u_i^n}{nu^n}+\sum_{j=\ell+1}^{\ell+k}\sum_{n\geq 1}\frac{u^n}{nu_j^n} \nonumber
\\&=\ln (A(x))^{-1}+\sum_{n\geq 1}\frac{1}{nu^n}\sum_{i=1}^{\ell}u_i^n
+\sum_{n\geq 1}\frac{u^n}{n}\sum_{j=\ell+1}^{\ell+k}\frac{1}{u_j^n}.\label{HHHOne2}
\end{align}
On the right hand side of Equation \eqref{HHHOne2}, the second term only contains negative power terms of $u$, and the third term only contains positive power terms of $u$. Thus the first term $\ln (A(x))^{-1}$ is exactly the constant term, denoted $\CT_u$, and we have
\begin{align*}
\ln (A(x))^{-1}&=\mathop{\mathrm{CT}}_u \ln\left(1-\frac{x(u+1)^{k+\ell}}{u^{\ell}}\right)^{-1}=\mathop{\mathrm{CT}}_u \sum_{n\geq 1}\frac{1}{n}\left(\frac{(1+u)^{k+\ell}\cdot x}{u^{\ell}} \right)^n\\
&=\mathop{\mathrm{CT}}_u\sum_{n\geq 1}\frac{(1+u)^{kn+\ell n}}{nu^{\ell n}}\cdot x^n=\sum_{n\geq 1}\binom{kn+\ell n}{\ell n}\frac{x^n}{n}.
\end{align*}
Therefore,
$$Q(x)=(A(x))^{-1}=\exp \left(\sum_{n\geq 1}\binom{kn+\ell n}{\ell n}\frac{x^n}{n} \right).$$
This completes the proof.
\end{proof}

By observing the second term on the right hand side of Equation \eqref{HHHOne2}, we found that
$$\sum_{m\geq 1}\frac{1}{mu^m}\sum_{i=1}^{\ell}u_i^m=\sum_{m\geq 1}\frac{p_m(u_1,...,u_{\ell})}{m}u^{-m},$$
where $p_m$ is the power sum symmetric function. Thus we can also extract an interesting formula for $p_m(u_1,...,u_{\ell})$ as follows. $m^{-1}p_m(u_1,...,u_{\ell})$ is the coefficient of $u^{-m}$ in Equation \eqref{HHHOne2}. Therefore we have
\begin{small}
\begin{align*}
m^{-1}p_m(u_1,...,u_{\ell})=[u^{-m}]\ln\left(1-\frac{x(u+1)^{k+\ell}}{u^{\ell}}\right)^{-1}
=[u^{-m}]\sum_{n\geq 1}\frac{(1+u)^{kn+\ell n}x^n}{nu^{\ell n}}
=\sum_{n\geq 1}\binom{kn+\ell n}{\ell n-m}\frac{x^n}{n},
\end{align*}
\end{small}
where $\binom{kn+\ell n}{\ell n-m}=0$ for $\ell n<m$. We summarize the above results as the following proposition.

\begin{prop}\label{PropPPM}
Let $m\geq 1$. Let $w_1,...,w_{\ell}$ be the unique solutions of the equation $(w-1)^{\ell}-xw^{k+\ell}=0$ that are fractional power series. We have
$$p_m(w_1-1,...,w_{\ell}-1)=m\cdot \sum_{n\geq 1}\binom{kn+\ell n}{\ell n-m}\frac{x^n}{n},$$
where $\binom{kn+\ell n}{\ell n-m}=0$ for $\ell n<m$.
\end{prop}

\begin{thm}\label{Mainresult22}
Follow the notation in Theorem \ref{MainResinIntrou}. Let $f_1(n)$ be the number of lattice paths from $(0,0)$ to $(\ell n-1,kn)$ that never go above the path $(N^kE^{\ell})^{n-1}N^kE^{\ell-1}$. Then we have
$$F_1(x)=\sum_{n\geq 0}f_1(n)x^n =\exp \left(\sum_{n\geq 1}\binom{kn+\ell n-1}{\ell n-1}\frac{x^n}{n}\right).$$
\end{thm}
\begin{proof}
Rewrite Equation \eqref{ExponentialW-11x} as
$$(w-1)^{\ell}-xw^{k+\ell}=B(x)\cdot \prod_{i=1}^{\ell}(w-w_i)\prod_{j=\ell+1}^{\ell+k}\left(1-\frac{w}{w_j}\right),$$
where
$$B(x)=(-1)^{k+1}x\cdot w_{\ell+1}\cdots w_{\ell+k}.$$
Thus we obtain
$$1-\frac{xw^{k+\ell}}{(w-1)^{\ell}}=B(x)\cdot \prod_{i=1}^{\ell}\frac{w-w_i}{w-1}\prod_{j=\ell+1}^{\ell+k}\left(1-\frac{w}{w_j}\right).$$
Now making the substitution $w\to w^{-1}$ and simplifying gives
\begin{align}\label{HHHOne3}
1-\frac{xw^{-k}}{(1-w)^{\ell}}=B(x)\cdot \prod_{i=1}^{\ell}\frac{1-ww_i}{1-w} \prod_{j=\ell+1}^{\ell+k}\left(1-\frac{1}{ww_j}\right).
\end{align}
By Equation \eqref{HHHOne3}, we have
$$w^k(1-w)^{\ell}-x=B(x)\cdot \prod_{i=1}^{\ell}(1-w_iw)\prod_{j=\ell+1}^{\ell+k}\left(w-\frac{1}{w_j}\right).$$
Then by comparing the coefficients of $w^{k+\ell}$ on both sides of the above equation, we obtain
$$(-1)^{\ell}w^{k+\ell}=B(x)\cdot (-1)^{\ell}w^{\ell+k}\prod_{i=1}^{\ell}w_i.$$
Therefore we have
$$(B(x))^{-1}=w_1w_2\cdots w_{\ell}=F_1(x),$$
where the second equality follows by Corollary \ref{CaseF1Thm}.

Again, by taking logarithms in \eqref{HHHOne3}, we obtain
\begin{align}
\ln\left(1-\frac{xw^{-k}}{(1-w)^{\ell}}\right)^{-1}&=\ln (B(x))^{-1}+\sum_{i=1}^{\ell}\ln\left(\frac{1-ww_i}{1-w}\right)^{-1}
+\sum_{j=\ell+1}^{\ell+k}\ln\left(1-\frac{1}{ww_j}\right)^{-1}\nonumber
\\&=\ln (B(x))^{-1}+\sum_{i=1}^{\ell}\ln\left(1-\frac{w(w_i-1)}{1-w}\right)^{-1}
+\sum_{j=\ell+1}^{\ell+k}\ln\left(1-\frac{1}{ww_j}\right)^{-1}\nonumber
\\&=\ln (B(x))^{-1}+\sum_{i=1}^{\ell}\sum_{n\geq 1}\frac{1}{n}\left(\frac{(w_i-1)w}{1-w}\right)^n
+\sum_{j=\ell+1}^{\ell+k}\sum_{n\geq 1}\frac{1}{n(ww_j)^n}.\label{HHHOne5}
\end{align}
On the right hand side of Equation \eqref{HHHOne5}, the second term  only contains positive power terms of $w$, and the third term only contains negative power terms of $w$. Thus the first term is just the constant term
\begin{align*}
\ln (B(x))^{-1}=\mathop{\mathrm{CT}}\limits_{w}\ln\left(1-\frac{xw^{-k}}{(1-w)^{\ell}}\right)^{-1}
=\sum_{n\geq 1}\frac{x^n}{n}\mathop{\mathrm{CT}}\limits_{w}\frac{1}{w^{kn}(1-w)^{\ell n}}=\sum_{n\geq 1}\binom{kn+\ell n-1}{kn}\frac{x^n}{n}.
\end{align*}
Thus we have
$$F_1(x)=(B(x))^{-1}=\exp \left(\sum_{n\geq 1}\binom{kn+\ell n-1}{\ell n-1}\frac{x^n}{n} \right).$$
This completes the proof.
\end{proof}

\begin{proof}[Proof of Theorem \ref{PeriodicWallGF}]
The theorem is followed by Proposition \ref{TableauxLatticePartit}, Theorem \ref{MainResinIntrou} and Theorem \ref{FF1ExpIntrod}.
\end{proof}

Similar to Proposition \ref{PropPPM}, we can obtain the following result.

\begin{prop}
Let $m\geq 1$. Let $w_1,...,w_{\ell}$ be the unique solutions of the equation $(w-1)^{\ell}-xw^{k+\ell}=0$ that are fractional power series. Let $w_{\ell+1},...,w_{\ell+k}$ be the other $k$ solutions in the field of fractional Laurent series $\mathbb{C}^{\text{fra}}((x))$.
We have
$$p_m(w_{\ell+1}^{-1},...,w_{\ell+k}^{-1})=m\cdot \sum_{n\geq 1}\binom{\ell n+kn-m-1}{kn-m}\frac{x^n}{n},$$
where $\binom{kn+\ell n-m-1}{kn-m}=0$ for $kn<m$ or $kn+\ell n<m+1$.
\end{prop}

\section{Two Byproducts}
Our proof of Theorem \ref{MainResinIntrou} seems lengthy, so we tried to solve the system \eqref{system-equat} for $C_t$, $0\leq t\leq \ell$ in other ways.
This leads to two byproducts, namely, Corollary \ref{c-byproduct} and Theorem \ref{SFdetermidentify}. Simple proofs of the byproducts may give rise simple proofs of Theorem \ref{MainResinIntrou}.
\subsection{Recursive Operation}
In \cite{AnnadeMier}, de Mier and Noy obtained the full Tutte polynomials as follows:
$$\sum_{n\geq 0}A_{n}(z,y)x^n=\frac{-(z-w_1)\cdots (z-w_{\ell})}{(xz^{k+\ell}-(z-1)^{\ell})(y+w_1-yw_1)\cdots (y+w_{\ell}-yw_{\ell})}.$$
We only need the specialization at $y=1$. Let
\begin{align}\label{AoverlineZ}
\mathrm{\overline{A}}(z):=\sum_{n\geq 0}A_{n}(z,1)x^n=\frac{-(z-w_1)\cdots (z-w_{\ell})}{xz^{k+\ell}-(z-1)^{\ell}}.
\end{align}

\begin{thm}\label{ThmFReccParit}
Following the notation in Theorem \ref{MainResinIntrou}. Let $1\leq r\leq \ell-1$. Then we have
$$F_r(x)=1+\sum_{i=1}^{\ell-r}S(\ell-r,i)x\mathrm{\overline{B}}_1^{(i-1)}(1),$$
where
$$\mathrm{\overline{B}}_1(z)=(k+1)z^k\mathrm{\overline{A}}(z)
+z^{k+1}\mathrm{\overline{A}}^{\prime}(z),\ \ \ \mathrm{\overline{A}}(z)=\frac{-(z-w_1)\cdots (z-w_{\ell})}{xz^{k+\ell}-(z-1)^{\ell}},$$
the $S(\ell-r,i)$ are Stirling numbers of the second kind and $\mathrm{\overline{B}}^{(i-1)}_1(z)$ is a $(i-1)$-th degree differential operation for $\mathrm{\overline{B}}_1(z)$.
\end{thm}
\begin{proof}
Follow the notation in Subsection \ref{LatticePath41}. Let
$$\mathrm{\overline{B}}_i(z)=\sum_{n\geq 0}B_{1,n}x^n,\ \ \ 1\leq i\leq \ell.$$
We simultaneously calculate the generating functions for both sides of Equation \eqref{B1B2B3C0C1C2} at $y=1$. This gives the following
equations.
\begin{align*}
&\mathrm{\overline{B}}_1(z)=\frac{z^{k+1}\mathrm{\overline{A}}(z)-C_0}{z-1},\ \
\mathrm{\overline{B}}_2(z)=\frac{z\mathrm{\overline{B}}_1(z)-C_1}{z-1},\ \
\cdots,\ \
\mathrm{\overline{B}}_{\ell}(z)=\frac{z\mathrm{\overline{B}}_{\ell-1}(z)-C_{\ell-1}}{z-1}.
\end{align*}
By L'Hospital's rule, we have
$$\mathrm{\overline{B}}_1(z)=(k+1)z^k\mathrm{\overline{A}}(z)
+z^{k+1}\mathrm{\overline{A}}^{\prime}(z),\ \ \ \mathrm{\overline{B}}_i(z)=\mathrm{\overline{B}}_{i-1}(z)+z\mathrm{\overline{B}}^{\prime}_{i-1}(z),\ \ 2\leq i\leq \ell.$$
By the recursion for the Stirling numbers of the second kind $S(t,i)$ (see \cite[A008277]{Sloane23} or \cite[Chapter 1]{RP.Stanley}):
$$S(t,i)=i\cdot S(t-1,i)+S(t-1,i-1),\ \ S(1,1)=1,\ \ S(t,i)=0,(i>t),$$
we have
$$\mathrm{\overline{B}}_t(z)=\sum_{i=1}^tS(t,i)z^{i-1}\mathrm{\overline{B}}^{(i-1)}_1(z),\ \ 1\leq t\leq \ell,$$
where $\mathrm{\overline{B}}^{(i-1)}_1(z)$ is a $(i-1)$-th derivative of $\mathrm{\overline{B}}_1(z)$ with respect to $z$.
Furthermore, by $C_t=\sum_{n\geq 0}C_{t,n}x^n=\mathrm{\overline{B}}_t(1)$, $1\leq t\leq \ell-1$, we have
$$C_t=\sum_{i=1}^tS(t,i)\mathrm{\overline{B}}_1^{(i-1)}(1),\ \ 1\leq t\leq \ell-1.$$
Note that $C_0=\mathrm{\overline{A}}(1)$.

The proof is then completed by $F_r(x)=xC_{\ell-r}+1$ for $1\leq r\leq \ell$.
\end{proof}

\begin{cor}\label{CollFl-1}
Follow the notation in Theorem \ref{ThmFReccParit}. We have
$$F_{\ell-1}(x)=1+\left(\ell-1-\sum_{i=1}^{\ell}\frac{1}{1-w_i}\right)\prod_{i=1}^{\ell}(1-w_i).$$
\end{cor}
\begin{proof}
By Theorem \ref{ThmFReccParit}, we have
$$\mathrm{\overline{A}}(1)=\frac{-1}{x}\prod_{j=1}^{\ell}(1-w_j),\ \ \ \ \
\mathrm{\overline{A}}^{\prime}(1)=-\frac{1}{x}\prod_{j=1}^{\ell}(1-w_j)
\left(\sum_{i=1}^{\ell}\frac{1}{1-w_i}-\ell-k\right).$$
Moreover,
$$\mathrm{\overline{B}}_1(1)=(k+1)\mathrm{\overline{A}}(1)+\mathrm{\overline{A}}^{\prime}(1)
=\frac{1}{x}\left(\ell-1-\sum_{i=1}^{\ell}\frac{1}{1-w_i}\right)\prod_{j=1}^{\ell}(1-w_j).$$
Furthermore, by Theorem \ref{ThmFReccParit}, we have $F_{\ell-1}=1+x\mathrm{\overline{B}}_1(1)$. This completes the proof.
\end{proof}

Theorem \ref{MainResinIntrou} implies that
\begin{align*}
F_{\ell-1}(x)&=\sum_{i=0}^{\ell-2}(-1)^{i}\binom{1+i}{i}e_{i+2}(w_1,...,w_{\ell})
\\&=e_2(w_1,...,w_{\ell})-2e_3(w_1,...,w_{\ell})+\cdots +(-1)^{\ell-2}(\ell-1)e_{\ell}(w_1,...,w_{\ell}).
\end{align*}
This result can also be explained by Corollary \ref{CollFl-1} as follows.
\begin{align*}
F_{\ell-1}(x)&=1+(1-w_1)\cdots (1-w_{\ell})\left(\ell-1-\sum_{i=1}^{\ell}\frac{1}{1-w_i}\right)
\\&=1+(\ell-1)(e_0-e_1+e_2+\cdots+(-1)^{\ell}e_{\ell})-\left(\frac{\prod_{j=1}^{\ell}(1-w_j)}{1-w_1}+\cdots +\frac{\prod_{j=1}^{\ell}(1-w_j)}{1-w_{\ell}}\right)
\\&=1+(\ell-1)(e_0-e_1+\cdots+(-1)^{\ell}e_{\ell})-(\ell e_0-(l-1)e_1+\cdots +(-1)^{\ell-1}e_{\ell-1}+0\cdot e_{\ell})
\\&=e_2-2e_3+\cdots +(-1)^{\ell-2}(\ell-1)e_{\ell},
\end{align*}
where $e_i=e_i(w_1,...,w_{\ell})$ for $0\leq i\leq \ell$.

\begin{cor}\label{c-byproduct}
Follow the notation in Theorem \ref{ThmFReccParit}. Then we have
\begin{align*}
&\sum_{i=1}^{\ell-r}S(\ell-r,i)x\Big((k+1)z^k\mathrm{\overline{A}}_j(z)
+z^{k+1}\mathrm{\overline{A}}_j^{(1)}(z)\Big)^{(i-1)}\Big|_{z=1}
\\=&\left\{
             \begin{array}{ll}
             -1 &\ \text{if}\ \ j=0;\ \ \ \ \ \ \ \ \ \ \ \ \ \ \ \ \\
             0 &\ \text{if}\ \ 1\leq j\leq \ell-r;\ \ \ \ \ \ \\
            (-1)^{j+r-1-\ell}\binom{j-1}{j+r-1-\ell} &\ \text{if}\ \ \ell-r+1\leq j\leq \ell,
             \end{array}
\right.
\end{align*}
where $\mathrm{\overline{A}}_j(z)=\frac{(-1)^{j+1}z^{\ell-j}}{xz^{k+\ell}-(z-1)^{\ell}}$ and the $S(\ell-r,i)$ are Stirling numbers of the second kind.
\end{cor}
\begin{proof}
By Equation \eqref{AoverlineZ}, we have
$$\mathrm{\overline{A}}(z)=\sum_{j=0}^{\ell}\mathrm{\overline{A}}_j(z)e_j(w_1,...,w_{\ell}).$$
Therefore
$$\mathrm{\overline{B}}_1(z)=(k+1)z^k\mathrm{\overline{A}}(z)
+z^{k+1}\mathrm{\overline{A}}^{\prime}(z)=\sum_{j=0}^{\ell}((k+1)z^k\mathrm{\overline{A}}_j(z)
+z^{k+1}\mathrm{\overline{A}}_j^{\prime}(z))e_j(w_1,...,w_{\ell}).$$
By Theorem \ref{ThmFReccParit}, we have
$$F_r(x)=1+\sum_{j=0}^{\ell}\left(\sum_{i=1}^{\ell-r}S(\ell-r,i)x\Big((k+1)z^k\mathrm{\overline{A}}_j(z)
+z^{k+1}\mathrm{\overline{A}}_j^{(1)}(z)\Big)^{(i-1)}\Big|_{z=1} \right)e_j(w_1,...,w_{\ell}).$$
The corollary is followed by Theorem \ref{MainResinIntrou}.
\end{proof}

\subsection{A Symmetric Function Formula}
It is natural to solve the system \eqref{system-equat} for $C_{\ell-r}$ using Kramer's rule. Rewrite \eqref{system-equat} as
\begin{align*}
\sum_{i=1}^{\ell}\left(\frac{w_j}{w_j-1}\right)^{i-1}xC_{\ell-i}=(w_j-1),\ \ 1\leq j\leq \ell.
\end{align*}
Let $\widehat{w}_i=\frac{w_i}{w_i-1}$, i.e., $w_i=\frac{\widehat{w}_i}{\widehat{w}_i-1}$ for $1\leq i\leq \ell$. By Kramer's rule, we have $xC_{\ell-r}=\frac{D_r}{D}$, where
\begin{align*}
D=\det\left(\begin{array}{cccc}
1 & \widehat{w}_1 &\cdots & \widehat{w}_1^{\ell-1}  \\
1 & \widehat{w}_2 &\cdots & \widehat{w}_2^{\ell-1} \\
\vdots &\vdots &\ddots & \vdots  \\
1 & \widehat{w}_{\ell} &\cdots & \widehat{w}_{\ell}^{\ell-1}
\end{array}\right)
=\det V(\widehat{w}_1,\widehat{w}_2,...,\widehat{w}_{\ell}).
\end{align*}
and
\begin{align*}
D_r=\det\left(\begin{array}{cccccccc}
1 & \widehat{w}_1 &\cdots & \widehat{w}_1^{r-2} & w_1-1 & \widehat{w}_1^{r} & \cdots & \widehat{w}_1^{\ell-1} \\
1 & \widehat{w}_2 &\cdots & \widehat{w}_2^{r-2} & w_2-1 & \widehat{w}_2^{r} & \cdots & \widehat{w}_2^{\ell-1} \\
\vdots &\vdots &\ddots & \vdots & \vdots & \vdots & \ddots  & \vdots \\
1 & \widehat{w}_{\ell} &\cdots & \widehat{w}_{\ell}^{r-2} & w_{\ell}-1 & \widehat{w}_{\ell}^{r} & \cdots & \widehat{w}_{\ell}^{\ell-1}
\end{array}\right).
\end{align*}
Now rewrite
$$w_i-1=-\frac{\widehat{w}_i}{1-\widehat{w}_i}-1=-\sum_{m\geq 0}(\widehat{w}_i)^m,\ \ \ 1\leq i\leq \ell,$$
and put in the above equation for $D_r$ and expand by linearity.
When $0\leq m\leq \ell$, the corresponding term is $0$, except that the term corresponding to $m=r-1$ is $-\det V(\widehat{w}_1,\widehat{w}_2,...,\widehat{w}_{\ell})$. The remaining terms correspond to $m\geq \ell$.
By Lemma \ref{SchurDeterm} and some simple calculations, we have
$$D_r=-\det V(\widehat{w}_1,...,\widehat{w}_{\ell})
+(-1)^{\ell-r+1}\det V(\widehat{w}_1,...,\widehat{w}_{\ell})\sum_{m\geq l}s_{(m-\ell+1,1^{\ell-r})}(\widehat{w}_1,...,\widehat{w}_{\ell}),$$
where $(m-\ell+1,1^{\ell-r})=(m-\ell+1,1,1,...,1)$ with length $\ell-r+1$.
Therefore, we have

\begin{equation}\label{FrKramerS}
F_r(x)=1+xC_{\ell-r}=1+\frac{D_r}{D}=(-1)^{\ell-r+1}\sum_{m\geq 1}s_{(m,1^{\ell-r})}(\widehat{w}_1,...,\widehat{w}_{\ell}).
\end{equation}
By Theorem \ref{MainResinIntrou}, we complete the proof of Theorem \ref{SFdetermidentify} as follows.
\begin{thm}\label{SFdetermidentify}
Let $\ell,r \in \mathbb{P}$  and $\ell\geq r$. Let $\widehat{w}_i=\frac{w_i}{w_i-1}$, $1\leq i\leq \ell$. We have
$$\sum_{m\geq 1}s_{(m,1^{\ell-r})}(\widehat{w}_1,...,\widehat{w}_{\ell})
=\sum_{i=0}^{r-1}(-1)^{\ell-r+i+1}\binom{\ell-r+i}{i}e_{\ell-r+i+1}(w_1,...,w_{\ell}).$$
\end{thm}

\begin{exa}
Suppose $\ell=r$. Theorem \ref{MainResinIntrou} implies that
\begin{equation}\label{IntrodFlllll}
F_{\ell}(x)=\sum_{i=0}^{\ell-1}(-1)^ie_{i+1}(w_1,...,w_{\ell}).
\end{equation}
This formula can be explained by Equation \eqref{FrKramerS} as follows.
$$F_{\ell}(x)=-\sum_{m\geq 1}s_m(\widehat{w}_1,...,\widehat{w}_{\ell}).$$
Because $s_m(x_1,x_2,...)=h_m(x_1,x_2,...)$ and $\sum_{m\geq 0}h_m(x_1,x_2,...)t^m=\frac{1}{\prod_{i}(1-x_it)}$ (see, e.g., \cite[Theorem 7.6.1]{RP.Stanley99}), we have
\begin{align*}
F_{\ell}(x)&=-\sum_{m\geq 1}h_m(\widehat{w}_1,...,\widehat{w}_{\ell})=1-\sum_{m\geq 0}h_m(\widehat{w}_1,...,\widehat{w}_{\ell})
\\&=1-\frac{1}{(1-\widehat{w}_1)(1-\widehat{w}_2)\cdots (1-\widehat{w}_{\ell})}
=1-(1-w_1)(1-w_2)\cdots (1-w_{\ell})
\\&=e_1(w_1,...,w_{\ell})-e_2(w_1,...,w_{\ell})+\cdots+(-1)^{\ell-1}e_{\ell}(w_1,...,w_{\ell}).
\end{align*}
\end{exa}

\section{Concluding Remark}
Let $S_m$ ($2\leq m\leq 6$) be the sequences [A079489], [A213403], [A213404], [A213405] and [A213406] in OEIS \cite{Sloane23}.
We have proved in Theorem \ref{PeriodicWallGF} that $S_m$ is in fact our $\overline{f}_m(n)$.
Now we can enrich the content of these sequences. We summarize as follows:

\begin{enumerate}
  \item The sequences $S_m$ are the number of rectangular Young tableaux $\mathcal{B}_m^n$ with periodic walls of a block $\mathcal{B}_m$.

  \item The sequences $S_m$ are the number of lattice paths from $(0,0)$ to $(mn,mn)$ that never go below the path $N(E^mN^m)^{n-1}E^mN^{m-1}$. (See Proposition \ref{DeterFormuFM}.)

  \item A recursive formula for sequences $S_m$ is shown in Proposition \ref{RecurFormula}.

  \item A determinant formula for sequences $S_m$ is given in Proposition \ref{DeterFormuFM}.

  \item The generating functions of sequences $S_m$ are given in Corollary \ref{PeriodicWallGF}. These are also the explanation of these sequences in OEIS.

  \item The generating functions of sequences $S_m$ are given by the exponential formulas in Theorem \ref{Mainresult22} with $k=l=m$.
\end{enumerate}

In this paper, we only enumerated Young tableaux with periodic walls over $\mathcal{B}_m$. The idea works for all Young building blocks with horizontal walls. Our next project is to study these generating functions.

%

\noindent
{\small \textbf{Acknowledgements:}
We are grateful to Menghao Qu for many useful suggestions.
This work is partially supported by the National Natural Science Foundation of China [12071311].

\end{document}